\documentclass{amsart}

\numberwithin{equation}{section}
\usepackage[matrix,arrow,curve,frame]{xy}
\usepackage{enumerate,calc}
\usepackage{amsmath,amsthm}
\usepackage{euscript}
\usepackage{amssymb}
\xymatrixcolsep{1.9pc}                  % Adjust size of diagrams.
\xymatrixrowsep{1.9pc}

\newtheorem{thm}[subsection]{Theorem}
\newtheorem{defn}[subsection]{Definition}
\newtheorem{prop}[subsection]{Proposition}

\newtheorem{cor}[subsection]{Corollary}
\newtheorem{lemma}[subsection]{Lemma}

\newtheorem{remark}[subsection]{Remark}

\newcommand{\cC}{{\cat C}}

\theoremstyle{definition}
\newtheorem{example}[subsection]{Example}
\newcommand{\cat}{\mathcal}
\newcommand{\lra}{\longrightarrow}
\newcommand{\llra}[1]{\stackrel{#1}{\lra}} 

\newcommand{\Z}{\mathbb{Z}}

\DeclareMathOperator{\id}{id}

\DeclareMathOperator{\Hom}{Hom}
\DeclareMathOperator{\Baar}{Bar}
\DeclareMathOperator{\Ho}{Ho}

\DeclareMathOperator{\Ext}{Ext}

\newcommand{\ra}{\rightarrow}
\newcommand{\rf}{\right\rfloor}
\newcommand{\lf}{\left\lfloor}
\newcommand{\we}{\llra{\sim}}

\begin{document}

\title{Realizing modules over the homology of a DGA}

\author{Gustavo Granja}
\thanks{Supported in part by FCT Portugal through program POCI 2010/FEDER and grant
POCI/MAT/58497/2004.}
\address{Centro de An\'alise Matem\'atica, Geometria e Sistemas Din\^amicos, Departamento de Matem\'atica, Instituto Superior T\'ecnico, Tech. Univ. Lisbon, Portugal}
\email{ggranja@math.ist.utl.pt}
\author{Sharon Hollander}\thanks{Supported by a Golda Meir postdoctoral fellowship.}
\address{Einstein Institute of Mathematics, Hebrew University, Jerusalem, Israel}
\curraddr{Centro de An\'alise Matem\'atica, Geometria e Sistemas Din\^amicos, Departamento de Matem\'atica, Instituto Superior T\'ecnico, Tech. Univ. Lisbon, Portugal}
\email{sjh@math.ist.utl.pt}

\date{July 23, 2007}


\begin{abstract}
Let $A$ be a DGA over a field and $X$ a module over $H_*(A)$. Fix an
$A_\infty$-structure on $H_*(A)$ making it quasi-isomorphic to $A$. 
We construct an equivalence of categories 
between $A_{n+1}$-module structures on $X$ and length $n$ Postnikov systems
in the derived category of $A$-modules based on the bar resolution of $X$. 
This implies that quasi-isomorphism classes of $A_n$-structures on $X$ are in bijective correspondence with weak equivalence classes of rigidifications of the first $n$ terms of the 
 bar resolution of $X$ to a complex of $A$-modules.
The above equivalences of categories are compatible for different values of $n$. This
implies that two obstruction theories for realizing $X$ as the homology of an $A$-module coincide.
\end{abstract}

\maketitle

\section{Introduction}

Let $A$ be a differential graded algebra over a field $k$ and let $R=H_*(A)$ be its homology.
We say that an $R$-module $X$ is \emph{realizable} if there exists a differential graded 
module $M$ over $A$ with $H_*(M)\simeq X$. This paper deals with two obstruction theories for answering the question of whether or not a module is realizable. 

One obstruction theory is based on the theory of $A_n$-structures. 
In \cite{St}, Stasheff introduced a 
hierarchy of higher homotopy associativity conditions for multiplications on chain complexes.
An $A_2$-structure is just a bilinear multiplication $m_2$, while an $A_3$-structure is an $A_2$-structure together with a homotopy $m_3$ between the two ways of bracketing a 3-fold product. An
$A_\infty$-structure consists of a sequence of higher associating homotopies $m_n$ 
satisfying certain conditions (see Section 2 for the definitions and also \cite{Ke} for an excellent introduction to the theory of $A_\infty$-algebras and modules).

Kadeishvili proved \cite{Ka} that there is a an $A_\infty$-structure on $H_*(A)$ making it quasi-isomorphic  to $A$ as an $A_\infty$-algebra. Such an equivalence induces an equivalence of derived categories of $A_\infty$-modules and the derived category of (homologically unital) $A_\infty$-modules over $A$ is equivalent to the usual derived category of DG modules over $A$. This implies that a module $X$ is realizable if and only if it admits the structure of an $A_\infty$-module over $H_*(A)$. The $H_*(A)$-module structure on $X$ makes it an $A_2$-module over the $A_\infty$-algebra $H_*(A)$ and so the problem of realizability is naturally broken down into the problem of extending an $A_n$-module structure on $X$ to an $A_{n+1}$-structure for successive $n$.

Given an $A_n$-structure on $X$, the obstruction to extending the underlying $A_{n-1}$-structure to an $A_{n+1}$-structure lies in $\Ext^{n,n-2}(X,X)$ (see Corollary \ref{hochschild}).
The original motivation for this paper was the observation that this first obstruction, i.e.
the obstruction to extending the given $A_2$-structure to an $A_4$-structure,  
is the primary obstruction to realizability described by other means in a recent paper of Benson, Krause and Schwede \cite{BKS}.

In \cite[Appendix A]{BKS}, the authors describe a general obstruction theory for realizability based
on the notion of a Postnikov system (see Definition \ref{postnikovdef}). This approach has its roots in stable homotopy theory (see \cite{Wh} for example). The basic idea is the following.
Since free modules and maps between them are clearly realizable, we can realize a free resolution for $X$ in the derived category of $A$-modules. The problem is then whether this "chain complex up to homotopy" can be rigidified to an actual chain complex of $A$-modules. This is explained in detail in Appendix \ref{appen} (see also \cite{BKS}). The Postnikov system approach can be applied more generally \cite[Appendix A]{BKS} to the problem of realizing modules over endomorphism rings of compact objects in triangulated categories.

Benson, Krause and Schwede define a canonical Hochschild class\footnote{The class in \cite{BKS} is 
actually in $HH^{3,-1}(H^*(A))$. This is because of our convention (which follows
\cite{We}) that the $k$-th shift $X[k]$ is the $k$-th desuspension of $X$ in the derived category, while in \cite{BKS} it is the $k$-th suspension.}  $\gamma_A \in HH^{3,1}(H_*(A))$ and show that the primary obstruction to building a Postnikov system for $X$ is the cup product $\id_X \cup \gamma_A  \in \Ext^{3,1}(X,X)$. Looking more closely one sees that the cocycles representing $\gamma_A$ are (up to sign) precisely the $A_3$-structures on $H_*(A)$ which extend to an $A_\infty$-structure quasi-isomorphic to $A$. It turns out that the primary obstruction $\id_X \cup \gamma_A$ to realizing $X$ is the obstruction to putting an $A_4$-module structure on $X$ (over any of these quasi-isomorphic $A_\infty$-structures on $H_*(A)$) and so the primary obstructions to realizing a module coincide from the two points of view. A natural question is then whether the two obstruction theories described above coincide in general. We answer this question in the affirmative.

We construct a functor from $A_n$-modules over $H_*(A)$ to filtered differential
 graded $A$-modules. This filtration gives rise to an $(n-1)$-Postnikov 
 system for $X$ and our main result is that this functor induces an equivalence of categories. 

\begin{thm}
\label{main}
Let $A$ be a differential graded algebra over a field.
There is an equivalence of categories between minimal $A_n$-modules over $H_*(A)$
and $(n-1)$-Postnikov systems based on the bar resolution\footnote{See Definition \ref{postnikovdef}.}.
\end{thm}

This is proved below as Corollary \ref{mainequiv} and Theorem \ref{mainequiv2}. It implies
in particular that for a fixed $H_*(A)$-module $X$, the moduli groupoid of $A_n$-structures on
$X$ is equivalent to the groupoid of $(n-1)$-Postnikov systems based on the bar resolution
for $X$ with isomorphisms which are the identity on the bar resolution (see Corollary \ref{mainequiv3}). These equivalences of categories are compatible with the forgetful functors for varying $n$ and hence yield an equivalence of the two obstruction theories (see Theorem \ref{bijectivecorr}).

It is a folk theorem in homotopy theory that given a chain complex $B_\bullet$ in a homotopy category $\Ho(\cC)$, $n$-Postnikov systems based on $B_\bullet$ correspond to rigidifications of the first $n$-terms of $B_\bullet$ to a chain complex in $\cC$. We explain this in Appendix \ref{appen}.
A consequence of this is the following result which is a direct consequence of Corollary \ref{mainequiv} and Proposition \ref{rigidifications}.
\begin{cor}
Isomorphism classes of $A_n$-module structures on $X$ are in bijective correspondence with 
weak equivalence classes of rigidifications of the complex formed by the first $n$ terms of the bar resolution for $X$.
\end{cor}

It would be interesting to know to what extent the previous results generalize to an abstract
stable homotopy theoretic context.

The equivalence of Theorem \ref{main} is reminiscent of the equivalence originally established by Stasheff \cite{St} between the existence of an $A_n$-structure (defined as a certain diagram in the homotopy category) and an $A_n$-form on a topological space $X$. It would be interesting to 
understand the precise relation.

A related problem is to give a similar characterization of $A_n$-algebra structures on a graded
algebra $R$. For such an $A_n$-structure, the obstruction to extending the underlying
$A_{n-1}$-structure to an $A_{n+1}$-structure is a class in the Hochschild cohomology group
$HH^{n+1,n-1}(R)$ (see Proposition \ref{hochschildalgebra}). It would be nice to have a description of these obstructions in terms of rigidification of diagrams.

\subsection{Conventions and Notation}

$k$ denotes a ground field. We always deal with (not necessarily bounded) homological complexes of $k$-vector spaces (unlike \cite{Ke} and \cite{BKS} who consider cohomological complexes).

We use the sign conventions of \cite{Ke} so that if $f,g$ are maps of graded modules, then
\[ (f\otimes g)(x\otimes y) = (-1)^{|f||x|} f(x)\otimes g(y). \]
In particular, this implies the commutation rule
\[ (f\otimes g)\circ (h \otimes j) = (-1)^{|g||h|} (f \circ h)\otimes(g\circ j). \]

If $C$ is a graded module we write $C[n]$ for the $n$-fold desuspension of $C$, i.e.
\[ C[n]_k = C_{k+n}. \]
If $C$ is a chain complex with differential $d \colon C[1] \to C$, then $C[n]$ is also a chain
complex with differential given by $(-1)^n d$.

If $f \colon C \to D$ is a map of chain complexes, the standard model for the 
\emph{homotopy fiber} of $f$ is the map $F\xrightarrow{\pi} C$ where $F$ is the 
complex (the desuspension of the mapping cone) defined by 
\[ F_k = D_{k+1} \oplus C_k \]
with differential given by the matrix 
\[  \begin{bmatrix} -d & f \\ 0 & d \end{bmatrix}, \]
and $\pi$ is the projection onto the second summand.

We will write $\lf x \rf$ for the greatest integer less than or equal to $x$.

Our differential graded algebras and modules are all unital.

\subsection{Organization of the paper}
In Section 2 we review the definition of $A_n$-algebra and module structures. In Section 3
we explain the obstruction to extending an $A_n$-module or algebra structure to an $A_{n+1}$-structure. In Section 4 we define the bar construction for an $A_n$-module. In Section 5
we recall the definition of Postnikov system and use the bar construction of the previous
section to produce an equivalence of categories between $A_n$-modules over $H_*(A)$ and 
Postnikov systems based on the bar construction. There is one appendix where we explain the 
relation between Postnikov systems and rigidifying complexes in the homotopy category.

\subsection{Acknowledgements}
We would like to thank the referee for a careful reading of the paper and many helpful 
suggestions and corrections.

\section{$A_n$-structures}

In this section we recall some basic definitions and notation regarding $A_n$-structures
\cite{Ke,St} and point out some simplifications of the formulas that take place in our 
setting. 

\begin{defn}
For $1 \leq n \leq \infty$, an \emph{$A_n$-algebra structure} on a 
graded $k$-module $R$ consists of maps
$$m_k: R^{\otimes k} \ra R[k-2] \quad \quad 1\leq k \leq n$$
satisfying the following relations for $m\leq n$
\begin{equation}
\label{ainftyalgebra}
\sum_{r+s+t=m} (-1)^{r+st}  m_{r+t+1} \circ (1^{\otimes r} \otimes m_{s} \otimes 1^{\otimes t})=0.
\end{equation}
An $A_n$-algebra is said to be \emph{minimal} if $m_1=0$.
\end{defn}

\begin{defn}
\label{defnAnmodule}
Let $R$ be an $A_n$-algebra. A \emph{right $A_l$-module $X$ over $R$} with $l\leq n$ consists of 
a graded $k$-module $X$, together with maps
$$m_k^X: X \otimes R^{\otimes k-1} \ra X[k-2]  \quad \quad 1\leq k\leq l$$
satisfying, for each $m\leq l$,
\begin{equation}
\label{ainftymodule}
\sum_{r+s+t=m} (-1)^{r+st} m_{r+t+1}^X(1^{\otimes r}\otimes m_s \otimes 1^{\otimes t})=0 
\end{equation}
where $m_s$ denotes $m_s^X$ whenever $r=0$ and $m_s^R$ otherwise. 

A \emph{left $A_l$-module $X$ over $R$} consists of maps $m_k^X \colon R^{\otimes k-1} \otimes
X \to X[k-2]$ satisfying \eqref{ainftymodule} with $m_s$ now denoting $m_s^X$ whenever $t=0$ and $m_s^R$ otherwise.

An $A_l$-module is said to be \emph{minimal} if $m_1^X = 0$.
\end{defn}
Note that $m_1m_1=0$ so that an $A_l$-module is a complex.

\begin{defn}
Let $R$ be an $A_n$-algebra and $l\leq n$.
A \emph{morphism of $A_l$-modules} over $R$, $f\colon X \ra Y$, consists of maps of $k$-modules 
\[ X\otimes R^{\otimes k-1} \llra{f_k} Y[k-1] \quad \quad 1\leq k\leq l \] 
satisfying the following equation in $\Hom(X\otimes R^{\otimes m-1},Y[m-2])$
for each $m\leq l$:
\begin{equation}
\label{ainftymor}
\sum_{r+s+t=m} (-1)^{r+st}f_{r+t+1}(1^{\otimes r}\otimes m_s \otimes 1^{\otimes t}) 
	= \sum_{i+j=m} (-1)^{(i+1)j} m_{j+1} (f_i\otimes 1^{\otimes j}) 
\end{equation}
where again $m_s$ denotes $m_s^X$ whenever $r=0$ and $m_s^R$ otherwise.

A morphism is called a \emph{quasi-isomorphism} if $f_1$ is a 
quasi-isomorphism. 
\end{defn}

\begin{defn}
A \emph{morphism of $A_n$-algebras} $f\colon A \ra B$ consists of maps 
\[ A^{\otimes k} \llra{f_k} B[k-1] \quad \quad k\leq n\] 
satisfying the following equation in $\Hom(A^{\otimes m},B[m-2])$ for each $1\leq m\leq n$:
\begin{equation}
\label{ainftymap}
\sum_{r+s+t=m} (-1)^{r+st}f_{r+t+1}(1^{\otimes r}\otimes m_s \otimes 1^{\otimes t}) = 
\sum (-1)^v m_u (f_{i_1}\otimes \ldots \otimes f_{i_u})
\end{equation}
where, on the right hand side, the sum is over all decompositions
$i_1+i_2+ \dots + i_u =m$ and $v=(u-1)(i_1 -1) +(u-2)(i_2 -1)+ \ldots + 2(i_{u-2}-1)
+ (i_{u-1} -1)$.

A morphism is called a \emph{quasi-isomorphism} if $f_1$ is a 
quasi-isomorphism.
\end{defn}

The $A_\infty$-algebras we will be dealing with arise from the following result.
\begin{thm}[Kadeishvili]
\label{Kadeishvili}
Let $A$ be a differential graded algebra over a field $k$. There is an $A_\infty$ structure on 
$H_*(A)$ together with a quasi-isomorphism of $A_\infty$-algebras
\[ f\colon H_*(A) \to A. \]
Moreover $m_1^{H_*(A)}=0$ and $m_2^{H_*(A)}$ is the associative multiplication induced
by the multiplication on $A$.
\end{thm}
The $A_\infty$-structure on $H_*(A)$ and the quasi-isomorphism $f$ are constructed 
inductively. The induction is started by picking a splitting $f_1$ for the projection 
$Z(A) \to H_*(A)$ and then making use of the following simplification of \eqref{ainftymap}.

\begin{lemma} 
\label{simpmorph}
Let $B$ be a minimal $A_\infty$-algebra and $A$ a differential graded
algebra. Then an $A_\infty$-morphism $f\colon B \to A$ consists of maps $f_n\colon B \to A[n-1]$
satisfying the following formulas for all $n$:
\begin{eqnarray*}
m_1 f_n & = & f_1(m_n) + \sum_{0<u+t<n-1} (-1)^{u +t(n-u-t)} f_{u+t+1}(1^{\otimes u} \otimes m_{n-u-t} \otimes 1^{\otimes t})+ \\ 
& + & \sum_{i=1}^{n-1} (-1)^i m_2(f_i \otimes f_{n-i})
\end{eqnarray*}
\end{lemma}

\begin{remark} \label{module-on-A}
In the situation of Theorem \ref{Kadeishvili}, writing $d$ for the differential and 
$\mu$ for the multiplication on $A$, the maps 
\[ m_n^A = (-1)^n \mu \circ(f_{n-1}\otimes 1) \colon H_*(A)^{\otimes n-1}\otimes A \to A[n-1]  \quad \text{ for } n\geq 2 \]
together with $m_1=d$ make $A$ a left $A_\infty$-module over $H_*(A)$.
\end{remark}

Given a differential graded algebra $A$, {\bf we fix an $A_\infty$-algebra structure on $H_*(A)$
and a quasi-isomorphism $f\colon H_*(A) \to A$} given by Theorem \ref{Kadeishvili} for the rest
of the paper.

A graded module $X$ over the associative graded algebra $H_*(A)$ has a natural $A_2$-structure
with $m_1^X=0$ and $m_2^X$ given by the action of $H_*(A)$. Moreover, an arbitrary map  
$m_3^X \colon X \otimes H_*(A)^{\otimes 2} \to X[1]$ gives $X$ the structure of an $A_3$-module
over $H_*(A)$.

The formulas \eqref{ainftymodule} defining an $A_n$-module simplify for modules
over $H_*(A)$ whose underlying $A_2$-structure arises from the situation described in the
previous paragraph. 
For example, an $A_2$-module structure consists of a map $m_2^X \colon X \otimes H_*(A)
\to X$ satisfying no hypothesis and an $A_3$-structure consists of an associative 
action $m_2^X$ together with a map $m_3^X \colon X \otimes H_*(A)^{\otimes 2} \to X[1]$
satisfying no hypothesis. More generally, an $A_n$-structure puts no restriction on the 
map $m_n^X$. We state here the simplified formulas here as they will be used repeatedly.

\begin{lemma}
\label{simplifmod}
Let $X$ be a graded vector space. An $A_n$-module structure on $X$ over $H_*(A)$ 
with $m_1^X=0$ consists of maps 
\[ m_k^X \colon X \otimes H_*(A)^{\otimes (k-1)} \to X[k-2]  \quad \quad 2\leq k \leq n\]
satisfying the following equations for $2\leq k<n$:
\[
\sum_{\stackrel{2\leq r+t+1 \leq k}{r+s+t=k+1}} 
(-1)^{r+st+1}  m_{r+t+1}^X \circ (1^r \otimes m_s \otimes 1^t)=0.
\]
\end{lemma}

\section{The obstruction to extending an $A_n$-structure to an $A_{n+1}$-structure}

Let $R$ be a minimal $A_\infty$-algebra over the field $k$ (so $m_2$ makes $R$ an associative algebra)
and let $X$ be a minimal right $A_n$-module over $R$ (so if $n\geq 3$, $X$ is in particular a module over the associative algebra $R$ in the usual sense). In this section we describe the 
set of $A_{n+1}$-structures on $X$ extending the given $A_n$-structure. We show that the
obstruction to the existence of an $A_{n+1}$-structure extending the underlying $A_{n-1}$-structure is an element in $\Ext^{n,n-2}_R(X,X)$.

The exact same computation shows that if $S$ is a graded algebra and one considers $A_n$-algebra
structures on $S$ extending the underlying $A_2$-structure then the obstructions to extension
are classes in the Hochschild cohomology $HH^{n,n-2}(S)$.

This obstruction theory for minimal $A_n$-algebras is also described in \cite[Appendix B.4]{Le}. Lef\`evre
also discusses obstruction theory for non-minimal $A_n$-algebras and modules in \cite[Appendix B]{Le} but
this does not seem relevant to the minimal case we consider here.

\bigskip

Recall the bar resolution of a module $M$ over a graded algebra $R$ (see for instance \cite[8.6.12]{We}):
\[ 
\cdots \to M\otimes_k R\otimes_k R \to M \otimes_k R \to M.
\]
This is a free resolution of $M$ as a right $R$-module. If $N$ is another right $R$-module we 
write
\[ \left( \Baar^{\star,*}(M,N), \partial\right)= \Hom^*( M\otimes R^{\otimes(\star+1)}, N) \]
for the induced cochain complex of graded $k$-modules. The cohomology of this complex
is $\Ext^{\star,*}_R(M,N)$.

If $R$ is a graded $k$-algebra, $M$ a graded $k$-module, $N$ a right $R$-module, and 
$f: M \to N$ a map of graded $k$-modules, we write
\[  f*1 \colon M \otimes R \to N \]
for the canonical extension of $f$ to a map of $R$-modules. 

Given $R$ and $X$ satisfying our standing assumptions, the map
\[ m_n^X \colon X \otimes R^{\otimes(n-1)} \to X[n-2] \]
yields a map
\[ m_n^X*1 \in \Baar^{n-1,n-2}(X,X). \] 

In this section we will write 
\begin{equation}
\label{phieq}
\phi_n = -m_2^X(1\otimes m_n) + \sum_{\stackrel{2< r+t+1<n}{r+s+t=n+1}} (-1)^{r+st} 
m_{r+t+1}^X(1^{\otimes r}\otimes m_s \otimes 1^{\otimes t}).
\end{equation}
Note that the condition $\phi_n=0$ is precisely the $(n+1)$-st condition 
for $(m_2^X,\ldots,m_{n-1}^X,0,0)$ to be an $A_{n+1}$-structure on $X$ (see 
Definition \ref{defnAnmodule}).

\begin{lemma}
\label{lemmahoch}
Let $\partial$ be the coboundary operator in the complex $\Baar^{\star,*}(X,X)$. 
There is an $A_{n+1}$-structure on $X$ extending a given $A_n$-structure if and
only if the following equation holds:
\begin{equation}
\label{mnboundaryalt}
\partial(m_n^X*1) + \phi_n*1 =0.
\end{equation}
Furthermore, when an extension exists, the set of extensions is in 1-1 correspondence
with the set of $k$-module maps $m_{n+1}^X \colon X\otimes R^{\otimes n} \to X[n-1]$.
\end{lemma}
\begin{proof}
By Lemma \ref{simplifmod} we need to check that 
\begin{equation}
\label{eq}
\partial(m_n^X*1) = (-1)^nm_2^X(m_n^X\otimes 1)*1 + \sum_{i=0}^{n-1} (-1)^i m_n^X(1^{\otimes i}\otimes m_2\otimes 1^{\otimes(n-i-1)})*1.
\end{equation}
The map,
\[
\partial(m_n^X*1) \colon X \otimes R^{\otimes (n+1)} \to X[n-2] 
\]
is given by the formula
\begin{eqnarray*}
\partial(m_n^X*1)(x,\zeta_1,\ldots,\zeta_{n+1}) = 
m_n^X(x\zeta_1,\ldots,\zeta_n)\zeta_{n+1} - m_n^X(x,\zeta_1\zeta_2,\ldots,\zeta_n)\zeta_{n+1} + \ldots \\
  +  (-1)^{n-1}m_n^X(x,\zeta_1,\ldots,\zeta_{n-1}\zeta_n)\zeta_{n+1} + 
(-1)^n m_n^X(x,\zeta_1,\ldots,\zeta_{n-1})\zeta_n\zeta_{n+1}.
\end{eqnarray*}
On the other hand, applying the right hand side of \eqref{eq} to $(x,\zeta_1,\ldots,\zeta_{n+1})$ we
get 
\begin{eqnarray*}
(-1)^nm_n^X(x,\zeta_1,\ldots,\zeta_{n-1})\zeta_n\zeta_{n+1} + m_n^X(x\zeta_1,\ldots,\zeta_n)\zeta_{n+1} + \\
+ \ldots + (-1)^{n-1} m_n^X(x,\zeta_1,\ldots,\zeta_{n-1}\zeta_n)\zeta_{n+1}.
\end{eqnarray*}
\end{proof}

\begin{lemma}
\label{lemmacocycle}
If $X$ is an $A_n$-module over $R$,  $\phi_n*1 \in \Baar^{n,n-2}(X,X)$ is a cocycle.
\end{lemma}
\begin{proof}
It suffices to check the condition $(\phi_n*1) \circ \partial =0$ on module generators. Thus we need to check that
\begin{equation}
\label{newphieq}
\sum_{i=0}^n (-1)^i \phi_n \left(1^{\otimes i} \otimes m_2 \otimes 1^{\otimes n-i} \right) + (-1)^{n+1}\phi_n*1 = 0.
\end{equation}
In the course of this proof, through \eqref{newcomb}, we omit 
\[ \sum_{\substack{r+s+t=n+1 \\ 2\leq r+t+1<n \\ (r,s,t)\neq(0,n,1)}} \]
which is understood to precede each sum.

Expanding the left summand in expression \eqref{newphieq} we obtain
\begin{eqnarray}
\label{term1}
\sum_{0\leq i<r} (-1)^{r+st+i} m_{r+t+1}^X(1^{\otimes i}\otimes m_2 \otimes 1^{\otimes r-i-1}\otimes m_s 
\otimes 1^{\otimes t}) + \\
\label{term2}
\sum_{r\leq i < r+s} (-1)^{r+st+i} m_{r+t+1}^X(1^{\otimes r} \otimes m_s(1^{\otimes i-r} \otimes m_2 \otimes 1^{\otimes r+s-i-1}) \otimes 1^{\otimes t})  + \\
\label{term3}
\sum_{r+s \leq i\leq n} (-1)^{r+st+i} m_{r+t+1}^X(1^{\otimes r}\otimes m_s \otimes 1^{\otimes i-r-s}\otimes m_2 \otimes 1^{\otimes r+t+s-i-1}) 
\end{eqnarray}
Using the $A_{s+1}$-structure on the middle term \eqref{term2} we find it is equal to
\begin{equation} 
\label{newmiddleterm}
\sum_{\substack{j+k+l=s+1 \\ 2\leq j+l+1<s}} (-1)^{st+j+kl+1} m_{r+t+1}^X (1^{\otimes r}\otimes m_{j+l+1} (1^{\otimes j} \otimes m_k \otimes 1^{\otimes l}) \otimes 1^{\otimes t}) 
\end{equation}
Separating out the term where $(j,k,l)=(1,s,0)$ and combining it with the term \eqref{term1} we obtain
\[ 
\sum_{0 \leq i\leq r} (-1)^{r+st+i} m_{r+t+1}^X(1^{\otimes i} \otimes m_2 \otimes 1^{\otimes r+t-i})
(1^{\otimes r+1} \otimes m_s \otimes 1^{\otimes t}).
\]
Similarly combining the term of \eqref{newmiddleterm}
where $(j,k,l)=(0,s,1)$ with the term \eqref{term3} we obtain
\[ 
\sum_{r \leq i\leq r+t} (-1)^{r+st+s+i-1} m_{r+t+1}^X(1^{\otimes i} \otimes m_2 \otimes 1^{\otimes r+t-i})
(1^{\otimes r} \otimes m_s \otimes 1^{\otimes t+1}).
\] 
Together the last two expressions yield
\begin{equation}
\label{combined}
\sum_{0 \leq i \leq r+t} (-1)^{r+st+s+i-1} m_{r+t+1}^X(1^{\otimes i} \otimes m_2 \otimes 1^{\otimes r+t-i})
(1^{\otimes r} \otimes m_s \otimes 1^{\otimes t+1}).
\end{equation}
Using the $(r+t+2)$-module structure in \eqref{combined} yields 
\begin{equation}
\label{newcomb}
\sum_{\substack{a+b+c=r+t+2 \\ b>2}} (-1)^{r+st+s+a+bc} m_{a+c+1}^X (1^{\otimes a}\otimes m_b 
\otimes 1^{\otimes c}) \circ (1^{\otimes r} \otimes m_s \otimes 1^{\otimes t+1}).
\end{equation}
The terms in \eqref{newcomb} corresponding to $(a,b,c)=(1,r+t+1,0)$ cancel with those terms in
 \eqref{newmiddleterm} that remain (i.e. $(j,k,l) \not \in \{ (1,s,0),(0,s,1)\}$) and satisfy
 $(r,s)=(1,n)$.
On the other hand, the terms in \eqref{newcomb} corresponding to $(a,b,c)=(0,r+t+1,1)$ 
cancel with the term $(-1)^{n+1} \phi_n*1$ in \eqref{newphieq}.

We are left with showing that the terms 
\begin{equation}
\label{nnewcomb}
\sum_{\substack{r+s+t=n+1\\ 2<s<n}} \sum_{\substack{a+b+c=r+t+2 \\ 2<b\leq r+t}} (-1)^{r+st+s+a+bc} m_{a+c+1}^X (1^{\otimes a} \otimes m_b \otimes 1^{\otimes c})
\circ (1^{\otimes r}\otimes m_s \otimes 1^{\otimes t+1}).
\end{equation}
left from \eqref{newcomb} and 
\begin{equation}
\label{nnewmiddleterm}
\sum_{\substack{r+s+t=n+1\\ 2<s<n}} \sum_{\substack{j+k+l=s+1\\ 2<k<s}} (-1)^{st+1+j+kl} m_{r+t+1}^X (1^{\otimes r} \otimes m_{j+l+1}(1^{\otimes j} \otimes m_k \otimes 1^{\otimes l}) \otimes 1^{\otimes t}).
\end{equation}
left from \eqref{newmiddleterm} add up to zero. In \eqref{nnewcomb}, the terms 
of the form $m_{a+c+1}^X(1^{\otimes a} \otimes m_b \otimes 1^{\otimes r-a-b} \otimes m_s \otimes 1^{\otimes t+1})$ appear twice with opposite signs. Finally, the
remaining terms of \eqref{newmiddleterm}, which are of the form $m_{a+c+1}^X (1^{\otimes a}\otimes m_b(1^{\otimes r-a} \otimes m_s \otimes 1^{\otimes t+1-c}) \otimes 1^{\otimes c})$, cancel with \eqref{nnewmiddleterm}.
\end{proof}

\begin{remark}
We have observed that, by definition, $(m_2^X,\ldots,m_{n-1}^X,0)$ extends to an $A_{n+1}$-structure if and only if $\phi_n=0$. Writing $\tilde\phi_n=0$ for the $A_{n+1}$-structure equation that $(m_2^X,\ldots,m_{n-1}^X,m_n^X)$ must satisfy in order to extend, it is easy to check that $(\tilde\phi_n-\phi_n)*1$ is a cocycle in the bar complex. Thus, given an $A_n$-structure on $X$, the previous lemma implies that $\tilde \phi_n*1$ is also a cocycle.
\end{remark}

We summarize the previous arguments in the following statement.
\begin{prop}
\label{hochschild}
Let $X$ be an $A_n$-module over $R$. 
\begin{enumerate}[(a)]
\item The underlying $A_{n-1}$-structure on $X$ can be extended to an $A_{n+1}$-structure iff the class 
$[\phi_n*1] \in  \Ext^{n,n-2}_R(X,X)$ vanishes.
\item  If $[\phi_n*1]=0$, the set of $A_{n+1}$-structures on $X$ extending the underlying $A_{n-1}$-structure is in bijective correspondence with pairs of $R$-module maps
\[ \psi \colon X\otimes R^{\otimes n} \to X[n-2], \quad \xi \colon X \otimes R^{\otimes (n+1)} \to X[n-1] \]
such that 
\[ \partial(\psi) = \phi_n*1 \]
\end{enumerate}
\end{prop}
\begin{proof}
Statement (a) is the content of Lemmas \ref{lemmahoch} and \ref{lemmacocycle}. Statement (b)
follows immediately from the fact that equation \eqref{mnboundaryalt} is the only equation involving $m_n^X$ among the equations defining an $A_{n+1}$-structure on $X$.
\end{proof}

\begin{remark}
Let $A$ be a differential graded algebra. In \cite{BKS} the authors consider the problem of 
deciding whether an $H_*(A)$-module $X$ is the homology of an $A$-module. They define a Hochschild cohomology class $\gamma_A \in HH^{3,1}(H_*(A))$ and show that the first obstruction is 
\[ 1_X \cup \gamma_A \in \Ext_{H_*(A)}^{3,1}(X,X)\]
(see \cite[Corollary 6.3]{BKS}). The choice of a cocycle 
representing $\gamma_A$ precisely corresponds to the choice of $m_3^{H_*(A)}$ in 
the inductive proof of Kadeishvili's theorem (compare Lemma \ref{simpmorph} with 
\cite[Construction 5.1 and Remark 5.8]{BKS}). 

The special case of Proposition \ref{hochschild} when $n=3$ says that an $R$-module $X$ has an $A_4$-structure if and only if the map 
\[ (m_2^X(1\otimes m_3))*1 \]
is a coboundary in $\Baar^{3,1}(X,X)$.Thus the obstruction described in \cite{BKS} is exactly the obstruction to the existence of an $A_4$-structure on $X$.
\end{remark}

\begin{example}
\label{nonrealex}
This example amplifies on the example considered in \cite[7.3,7.4,7.6]{BKS}. Let $L=k[z]/z^n$ be the truncated polynomial algebra of height $n$ over a field $k$. Let $A$ be the endomorphism DGA of the complete resolution  $\hat{P}$ of the trivial $L$-module $k$. $\hat{P}$ is defined by $\hat{P}_i=L$ for each $i\in \Z$ with differentials 
$d_i \colon \hat{P}_i \to \hat{P}_{i-1}$ given by the formulas 
\[ d_i = \begin{cases} 
\text{multiplication by } -z^{n-1} & \text{ if } i \text{ is even,} \\
\text{multiplication by } z & \text{otherwise.}
\end{cases}
\]
Note that if $k$ has characteristic $p$ and $n$ is a power of $p$, $L$ is isomorphic to the group
algebra of the cyclic group $C_n$ and then the homology of $A$ is the Tate cohomology of $C_n$.

The homology algebra of $A$ is \cite[Theorem 7.3]{BKS}:
\[ H_*(A) = \begin{cases}
k[x^{\pm 1}] & \text{ if } n=2, \\
\Lambda(x) \otimes k[y^{\pm 1}] & \text{ if } n>2, 
\end{cases}
\]
where $\Lambda(x)$ denotes the exterior algebra on $x$, and $|x|=-1$ and $|y|=-2$. In the proof
the authors define the first two maps $f_1 \colon H_*(A) \to A$ and $f_2 \colon H_*(A) \to A[1]$ 
in a quasi-isomorphism of $A_\infty$-algebras $f\colon H_*(A) \to A$ (cf. Theorem \ref{Kadeishvili})
and use this to find the $A_3$-structure on $H_*(A)$, $m_3 \colon H_*(A)^{\otimes 3} \to H_*(A)[1]$ 
(in their terminology this is the Hochschild cocycle $m$ representing the canonical class as in
the previous remark).

For $n\neq 3$, $m_3$ vanishes, while for $n=3$, it is given by the formula
\[ 
m_3(a,b,c) = \begin{cases}
0 & \text{ if } |a|,|b|, \text{ or } |c| \text{ is even,} \\
y^{i+j+k+1} & \text{ if } a=xy^i,b=xy^j,c=xy^k.
\end{cases}
\]

Proceeding as in the proof of \cite[Theorem 7.3]{BKS} we can inductively find formulas for the 
remaining maps $f_i \colon H_*(A) \to A[i-1]$. From this we see that, in general, 
the $A_\infty$-structure on $H_*(A)$ consists of only $m_2$ and $m_n$ with all 
other $m_k$'s vanishing and $m_n$ given by the formula
\[
m_n(a_1,\ldots,a_n) = \begin{cases}
0 & \text{ if one of the } |a_i|\text{'s is even,} \\
y^{j_1+\ldots+j_n+1} & \text{ if } a_i=xy^{j_i}.
\end{cases}
\]

In the case when $n=3$ the authors show in \cite[Example 7.6]{BKS} that the realizable $H_*(A)$-modules
are precisely the free ones. For $n>3$, any $H_*(A)$-module $X$ admits a trivial $A_n$-structure with $m_k^X=0$ for $2<k\leq n$.
The argument of \cite[Example 7.6]{BKS} shows more generally that for this trivial $A_n$-structure
to extend to an $A_{n+1}$-structure, $X$ must be a free module. However, it is no longer true
that only free modules are realizable. In fact, for $n>3$, all modules are direct summands 
of realizable modules (by the previous calculation together with the main result of \cite{BKS}) but they are certainly not all direct summands of free modules.

For example, for $n>3$,
\[ H_*(\Hom(\hat{P},k[z]/z^2)) = k[y^{\pm 1}] \oplus k[y^{\pm 1}][1] \]
(with $x$ acting trivially) is obviously realizable. 

On the other hand, the $H_*(A)$-module 
\[ X=k[y^{\pm 1}] = H_*(A)/xH_*(A) \]
is not realizable. Indeed, for any
choice of $m_k^X \colon X \otimes H_*(A)^{\otimes (k-1)} \to X[k-2]$ and any $a\in X$ we have 
\[ m_k^X(a,x,\ldots,x) = 0 \]
since $X$ is concentrated in even degrees. It follows that $X$ can't be given an 
$A_{n+1}$-structure: when we evaluate
\[ \sum_{\substack{2\leq r+t+1\leq n \\ r+s+t=n+1}} (-1)^{r+st} m_{r+t+1}^X(1^{\otimes r} \otimes
m_s \otimes 1^{\otimes t}) (a,x,\ldots, x) \]
all terms except $a m_n(x,\ldots,x) = ay$ vanish (either because $x$ acts trivially on $X$, $x^2=0$, $m_k^X(a,x,\ldots,x)=0$, or because $m_k=0$ for $k<n$).

The algebra $\Ext^{\star,\ast}(X,X)$ is a polynomial algebra on $k[y^{\pm 1}]$ on a generator 
in bidegree $(1,-1)$ ($y$ has degree $(0,-2)$). The obstruction to extending the trivial $A_n$-module structure on $X$ to an $A_{n+1}$-structure must therefore be a generator of $\Ext^{n,n-2}_{H_*(A)}(X,X)$.

It is somewhat surprising that for the realizable module $Y=k[y^{\pm 1}] \oplus k[y^{\pm 1}][1]$
we can't choose $m_k^Y$ to vanish for $k<n$. One can check that an $A_\infty$-structure on $Y$
can be defined in the following way. Let $a$ and $b$ be module generators for $Y$ in degrees
$0$ and $-1$ respectively.

If $n$ is even, set (for $k\geq 3$)
\[
m_k^Y(m,xy^{i_2},\ldots,xy^{i_k}) = \begin{cases}
b y^{i_1+\ldots+i_k} & \text{ if } m=ay^{i_1} \text{ and } k=\tfrac{n}{2}-1,\\
a y^{i_1+\ldots+i_k} & \text{ if } m=by^{i_1} \text{ and } k=\tfrac{n}{2}-1,\\
0 & \text{ otherwise.}
\end{cases}
\]

If $n$ is odd, set
\[
m_k^Y(m,xy^{i_2},\ldots,xy^{i_k}) = \begin{cases}
b y^{i_1+\ldots+i_k} & \text{ if } m=ay^{i_1}, k \in\{\tfrac{n+1}{2},\tfrac{n+3}{2}\},
\text{ and } k \text{ is even,}\\
a y^{i_1+\ldots+i_k} & \text{ if } m=by^{i_1}, k \in\{\tfrac{n+1}{2},\tfrac{n+3}{2}\}, 
\text{ and } k \text{ is odd,}\\
0 & \text{ otherwise.}
\end{cases}
\]
\end{example}

The analog for algebras of Proposition \ref{hochschild} is the following. Consider the Hochschild 
complex
\[ (C^{n,m}(S) = \Hom^{m}_{S\otimes S^{op}} (S^{\otimes (n+2)}, S), \partial_H). \]
An element in $C^{n,m}(S)$ is represented by a map of vector spaces $f \colon S^{\otimes n} \to S$ of
degree $m$ and, in these terms, the differential is given by the formula
\[ \partial_H(f) =  m_2 (1\otimes f )- \sum_{j=0}^{n-1} (-1)^j f \circ (1^{\otimes j} \otimes m_2 \otimes 1^{\otimes n-j-1}) + (-1)^n m_2(f\otimes 1). \]
The analog of \eqref{mnboundaryalt} is that an $A_n$-algebra structure on $S$ can be extended
to an $A_{n+1}$-algebra structure iff the following equation is satisfied:
\[ \partial_H(m_n) = \sum_{\stackrel{2<r+t+1<n}{r+s+t=n+1}} (-1)^{r+st} m_{r+t+1}(1^{\otimes r} \otimes m_s \otimes 1^{\otimes t}) \]
and the same computations as in the proofs of Lemmas \ref{lemmahoch} and \ref{lemmacocycle} give the following analog of Proposition \ref{hochschild} which can also be found in \cite[Lemma B.4.1]{Le}.
\begin{prop}
\label{hochschildalgebra}
Let $S$ be a graded algebra. Given an $A_n$-structure on $S$ extending the given $A_2$-structure,
the underlying $A_{n-1}$-structure can be extended to an $A_{n+1}$-structure iff the Hochschild cocycle
\[ \sum_{\stackrel{2<r+t+1<n}{r+s+t=n+1}} (-1)^{r+st} m_{r+t+1}(1^{\otimes r} \otimes m_s \otimes 1^{\otimes t}) \]
represents the trivial class in $HH^{n+1,n-2}(S)$.  
\end{prop}

If $S$ is an $A_\infty$-algebra with $m_1=m_3= \ldots =m_{n-1}=0$ then $m_n: S^{\otimes n} \ra S$ is a Hochschild cocycle. 
The primary obstruction to realizing a module $X$ is then the obstruction to giving $X$ an 
$A_{n+1}$-structure, namely the class $1_X \cup [m_n] \in \Ext^{n,n-2}_S(X,X)$.
In fact, any $S$-module $X$ can be given an $A_{n-1}$-module structure with $m_3^X=\ldots =m_{n-1}^X=0$ and one can extend this to an $A_{n+1}$-module structure on $X$ if and only if 
\[  (m_2^X(1\otimes m_n))*1, \]
which is a cocycle representing $1_X \cup [m_n]$, is a coboundary in $\Baar^{n,n-2}(X,X)$.
This is exactly the situation for the non-realizable module $X$ in Example \ref{nonrealex}.

\begin{example}
Let $S=R[\epsilon]/\epsilon^2$ where $R$ is a $k$-algebra concentrated in degree $0$ and $|\epsilon|=n-2$. If $\{m_n\}$ is an $A_\infty$-structure on $S$ then for degree reasons $m_i=0$ for $i\neq 2,n$ and $m_n$ is determined by a $k$-linear map $R^{\otimes n} \to R$
which must be a Hochschild cocycle. One can check that two $A_\infty$-structures on $S$ are 
quasi-isomorphic iff the corresponding cocycles represent the same cohomology class (cf. \cite[3.2]{Ke}).
\end{example}

\section{The bar construction}

Recall that we have fixed an $A_\infty$-structure on $H_*(A)$ and 
a quasi-isomorphism $f\colon  H_*(A) \to A$, which in
turn gives $A$ the structure of an $A_\infty$-$H_*(A)$-module (see Remark \ref{module-on-A}).
The goal of this section is to construct a functor, denoted $B(-,H_*(A),A)$, 
from $A_\infty$-$H_*(A)$-modules to differential graded $A$-modules.  
The functor can be writen as a directed colimit of functors $B_{n-1}(-,H_*(A),A)$, 
from $A_n$-$H_*(A)$-modules to differential graded $A$-modules.

\bigskip

Given a (minimal) $A_n$-module structure on an $H_*(A)$-module $X$
\[  m_k^X \colon X \otimes H_*(A)^{\otimes (k-1)} \to X, \quad 2\leq k \leq n, \]
let $R_k$ denote the free differential graded $A$-module defined by
\[ R_k = X \otimes H_*(A)^{\otimes k} \otimes A. \]
For $1\leq l\leq k+1$, let 
\[ M_{k,l} \colon R_k \to R_{k-l+1}[l-2] \]
be defined as 
\[ 
M_{k,l} = \sum_{i=0}^{k+2-l} (-1)^{i(l-1)} 1^{\otimes i}\otimes m_l \otimes 1^{\otimes k-l-i+2} 
\]
where, in the first term of the sum, $m_l$ stands for $m_l^X$ and, in the last term, 
$m_l$ stands for $m_l^A = (-1)^l f_{l-1}*1$ if $l>1$ and for the differential $d$ on $A$ if $l=1$ (see
Remark \ref{module-on-A}).  We will sometimes write $D$ for $M_{k,2}$ and
$d$ for $M_{k,1}$.

The formulas in the following definition were obtained when attempting to construct a Postnikov
system associated to an $A_{n+1}$-module (see Theorem \ref{functorPost}).

\begin{defn}
Given an $A_{n+1}$-module $X$ over $H_*(A)$ (with $1\leq n \leq \infty$), the \emph{bar construction
on $X$} is the right $A$-module $B_{n}(X,H_*(A),A)$ defined by 
\[ 
	\oplus_{i=0}^{n} \bigl( X \otimes H_*(A)^{\otimes i} \otimes A\bigr)[-i]  
	= \oplus_{i=0}^{n} R_i[-i].
\]
The differential on $B_{n}(X,H_*(A),A)$ is defined on the summand $R_l$ by the 
following formula 
\begin{equation}
\label{formulafordifferential}
\partial_{|R_l}
 =  \sum_{i+j+k=l+2} (-1)^{k+j+ij+\lf\frac{j-1}{2}\rf} 1^{\otimes i} \otimes m_j \otimes 1^{\otimes k} 
=\sum_{j=1}^{l+1}  (-1)^{l+\lf\frac{j-1}{2}\rf} M_{l,j} 
\end{equation}
\end{defn}
We use $\lf x\rf$ to denote the greatest integer less than or equal to $x$. The following 
easily checked formula will be used constantly in computations.

\begin{lemma}
\label{niceformula}
For any integers $i$ and $j$ 
\[ \lf\frac{i+1}{2} \rf + \lf \frac j 2 \rf \equiv \lf \frac{j-i}{2}\rf + ij \bmod 2. \]
\end{lemma}
\begin{lemma}
\label{dgmodule}
The formulas \eqref{formulafordifferential} give $B_{n}(X,H_*(A),A)$ the structure
of a differential graded $A$-module.
\end{lemma}
\begin{proof}
It's easy to check that the Leibniz rule holds so it is enough to check that 
\eqref{formulafordifferential} defines a differential on $B_{n}(X,H_*(A),A)$.
The projection to $R_m$ of $\partial^2_{|R_l}$ is given by the formula
\begin{equation}
\label{summationthatvanishes}
 \sum_{j=1}^{l-m+1} (-1)^{1-j+ \lf\frac{j-1}{2}\rf + \lf\frac{l-j-m+1}{2}\rf} 
M_{l-j+1,l-j-m+2} M_{l,j}. 
\end{equation} 
By Lemma \ref{niceformula} the sign in the previous expression is equal to 
\[ (-1)^{ (l-m)(j-1) + \lf\frac{l-m}{2}\rf }. \]
Since $(-1)^{\lf\frac{l-m}{2}\rf}$ is independent of $j$, this factor can be eliminated
and the equation $\partial^2=0$ then follows from the relations that must be satisfied
because $H_*(A)$ is an $A_{n+1}$-algebra and $X$ and $A$ are $A_{n+1}$-modules over $H_*(A)$.
\end{proof}
We also need to explain the functoriality of the bar construction.  
\begin{prop}
\label{defnBn}
Let $g:X \ra Y$ be a map of $A_{n+1}$-modules (with $1\leq n \leq \infty$).
The map 
\[ B_n(g)\colon  B_n(X,H_*(A),A) \ra  B_n(Y,H_*(A),A) \]
defined by the matrix with entries
\begin{equation}
\label{bgformula}
B_n(g)_{i,j}= (-1)^{\lf\frac{j-i+1}{2}\rf}g_{j-i+1} \otimes 1^{\otimes i}
\end{equation} 
for $1\leq i \leq j \leq n+1$, or
$$\begin{bmatrix} g_1 \otimes 1 & -g_2 \otimes 1 & -g_3 \otimes 1 & g_4 \otimes 1 & \dots \\
0 &  g_1 \otimes 1^{\otimes 2} & -g_2 \otimes 1^{\otimes 2} & -g_3 \otimes 1^{\otimes 2} & 
\dots \\  0 & 0 &  g_1 \otimes 1^{\otimes 3} & -g_2 \otimes 1^{\otimes 3} & \dots \\
0 & 0 & 0 &  g_1 \otimes 1^{\otimes 4} & \dots \\
\dots & \dots & \dots & \dots & \dots 
\end{bmatrix}$$
is a map of differential graded $A$-modules.
\end{prop}
\begin{proof}
This computation is similar to the one above and hence is omitted.
\end{proof}

We also write $B(g) = B_{\infty}(g)$.

\begin{prop}
\label{functbarconst}
For $1 \leq n \leq \infty$, the assignments
\[ X \mapsto B_n(X,H_*(A),A) \quad \left(X \xrightarrow{g} Y\right) \mapsto B_n(g) \]
define a functor from $A_{n+1}-H_*(A)$-modules to differential graded $A$-modules.
\end{prop}
\begin{proof}
Matrix multiplication precisely corresponds to the composition of $A_{n+1}$-module maps as defined on \cite[p. 15]{Ke}.
\end{proof}

\begin{remark}
The quasi-isomorphism of $A_\infty$-algebras $f \colon H_*(A) \we A$ makes 
$A$ and $A_\infty$-$H_*(A)$-$A$-bimodule. Although the formula for the differential \eqref{formulafordifferential} is different, it seems likely that $B_\infty(-,H_*(A),A)$ is equivalent to the functor $-\stackrel{\infty}{\otimes}_{H_*(A)} A$ considered in \cite[Section 4.1, p.114]{Le}.
\end{remark}

\section{$A_n$-structures and Postnikov systems}

In this section we describe the obstruction theory to realizing a module based on the notion
of a Postnikov system \cite{BKS} and show that the bar construction of the previous section
gives us a functor from $A_{n+1}$-module structures to $n$-Postnikov systems. We then show
that the obstructions to extending an $A_{n+1}$-structure or its associated $n$-Postnikov system 
agree. It follows by induction that any Postnikov system arising from the bar resolution
of $X$ comes from an $A_{n+1}$-structure. Finally we prove that this assignment is fully
faithful completing the proof of Theorem \ref{main}.

\bigskip
In this section, we will often use the following simple formula for the maps in the derived category of $A$-modules when the source is free: if $V$ is a $k$-module and $N$ is a 
differential graded module over $A$ then
\[ [V \otimes A, N] = \Hom_{H_*(A)}(V\otimes H_*(A), H_*(N)). \]

\begin{defn}
\label{postnikovdef}
Let $A$ be a differential graded algebra and $X$ be an $H_*(A)$-module. An
\emph{$n$-Postnikov system} for $X$ is a commutative diagram in the derived category
of $A$-modules
\[\xymatrix{
Y_n \ar[d]_{j_n} &  Y_{n-1} \ar[d]_{j_{n-1}} & Y_{n-2} \ar[d]^{j_{n-2}} & & Y_1 \ar[d]_{j_1} \\
C_n \ar[ur]^{i_{n}} \ar[r]_{d_n} & C_{n-1} \ar[ur]^{i_{n-1}} \ar[r]_{d_{n-1}} & C_{n-2} & \cdots & C_1 \ar[r]_{i_1} & C_0 
}\]
satisfying 
\begin{enumerate}[(i)]
\item $j_k$ is the homotopy fiber of $i_{k}$ (i.e. $Y_k \to C_k \to Y_{k-1}$ is part of a triangle),
\item $C_k$ is a free $A$-module,
\item there is a map $H_*(C_0) \to X$ such that the following is an exact sequence $H_*(C_n) \to \cdots \to H_*(C_0) \to X \to 0$.
\end{enumerate}
Maps of $n$-Postnikov systems are maps of diagrams in the derived category which 
restrict to maps of triangles.

We say that an $n$-Postnikov system is \emph{based on the bar resolution}
if $H_*(C_\star)$ is isomorphic to the bar resolution for $X$. A map is \emph{based on the bar resolution} if the maps $H_*(C_k) \to H_*(C_k')$ are of the form $g\otimes 1^{\otimes(k+1)}$ with 
$g\colon X \to X'$ a map of $H_*(A)$-modules.
\end{defn}
\begin{remark}
The previous definition differs from the definition of $n$-Postnikov system in \cite[Definition A.6]{BKS} in that the homotopy fiber of $i_{n}$ is included in the diagram. This distinction
is only relevant when considering maps of Postnikov systems. 
\end{remark}
A simple diagram chase shows (see \cite[Lemma A.12]{BKS}) that an $n$-Postnikov system yields an
exact sequence
\[ 0 \to X[n-1] \to H_*(Y_{n-1}) \to H_*(C_{n-1}) \to \ldots \to H_*(C_0) \to X \to 0. \]
It follows from the proof of this result that the following diagram commutes
\begin{equation}
\label{diagram}
\xymatrix{
& H_*(C_0)[n-1] \ar[dr]^j \ar[d] & \\ 0 \ar[r] & X[n-1] \ar[r] \ar[d] & H_*(Y_{n-1}) \\
 & 0 & }
\end{equation}
where $j$ denotes the composite
\[ H_*(C_0)[n-1] \to H_*(Y_1)[n-2] \to \cdots \to H_*(Y_{n-2})[1] \to H_*(Y_{n-1}). \]
Note that when $C_*$ is the bar resolution, the nontrivial vertical map in 
\eqref{diagram} is the multiplication map $X\otimes H_*(A) \to X$.

\begin{thm}
\label{functorPost}
Let $X$ be an $A_{n+1}$-module over $H_*(A)$, $R_k = X \otimes H_*(A)^{\otimes k} \otimes A$ and 
$Y_k = B_k(X,H_*(A),A)[k]$. Then the following diagram of $A$-modules projects to an $n$-Postnikov system for $X$:
\[ 
\xymatrix{
Y_n \ar[d]_{\pi_n} & Y_{n-1} \ar[d]_{\pi_{n-1}} & Y_{n-2}\ar[d]_{\pi_{n-2}} & \ldots & Y_1 \ar[d]_{\pi_1} & \\
R_n \ar[ur]^{i_n} \ar[r]_{M_{n,2}} & R_{n-1} \ar[ur]^{i_{n-1}} \ar[r]_{M_{n-1,2}} & R_{n-2} & \ldots & R_1 \ar[r]_{M_{1,2}} & R_0. }
\]
Here $\pi_k$ denotes the projection onto the last summand and 
\[ 
i_k = \begin{bmatrix}
(-1)^{\lf\frac k 2\rf} M_{k,k+1} \\
\vdots \\
(-1)^{\lf\frac{j-1}{2}\rf} M_{k,j}\\ 
\vdots \\
M_{k,2}
\end{bmatrix} \]
This assignment is functorial with respect to maps of $A_{n+1}$-modules.
\end{thm}
\begin{proof}
By definition
\[ Y_k = R_0[k] \oplus \ldots \oplus R_{k-1}[1] \oplus R_k \]
as a graded $A$-module, and the $ij$-th entry ($1\leq i \leq j \leq k+1$) of the matrix $\partial_{Y_k}$ is 
\[ (-1)^{k-j+1+ \lf\frac{j-i}{2}\rf} M_{j-1,j-i+1}. \]
Therefore the differential on $Y_k$ satisfies the following inductive formula
\[ \partial_{Y_k} = \begin{bmatrix} -\partial_{Y_{k-1}} & i_k \\ 0 & d \end{bmatrix}. \]
It follows that $i_k$ is a map of differential graded modules because
this condition is precisely the condition that the upper right hand vector in the matrix 
$\partial_{Y_k}^2$ vanishes. Clearly $Y_k$ is the homotopy fiber of $i_k \colon R_k \to Y_{k-1}$.
Finally, functoriality follows from Proposition \ref{functbarconst}.
\end{proof}

\begin{defn}
The \emph{canonical $n$-Postnikov system} associated to an $A_{n+1}$-structure on $X$
is the Postnikov system defined in Theorem \ref{functorPost}.
\end{defn}

\begin{thm}
\label{bijectivecorr}
Let $X$ be an $A_{n+1}$-module over $H_*(A)$. There is a bijective correspondence between
the sets of
\begin{enumerate}[(i)]
\item $A_{n+2}$-structures $(m_2^X,\ldots,m_{n+1}^X,\phi)$ on $X$,
\item lifts in the homotopy category 
\[ \xymatrix{ & Y_n \ar[d]^{\pi_n} \\ R_{n+1} \ar@{-->}[ur]^j \ar[r]_{D} & R_n. } \]
\end{enumerate}
The assignment sends $(m_2^X,\ldots,m_{n+1}^X,\phi)$ to the homotopy class of the map $i_{n+1}$ defined in Theorem \ref{functorPost} from the $A_{n+2}$-structure.

In other words, an $A_{n+1}$-structure on $X$ extends one stage iff its associated canonical $n$-Postnikov system extends one stage and, in that case, the extensions are in bijective correspondence. 
\end{thm}

\begin{proof}
The canonical $n$-Postnikov system associated to the $A_{n+1}$-structure on $X$ extends if and only if the map $i_nD$ is null. Since $R_{n+1}$ is a free $A$-module, this is equivalent to $H_*(i_n D)$ being the zero map. As $X[n-1] \to H_*(Y_{n+1})$ is an inclusion, this amounts to the vanishing of the map $\overline{i_n D}$ in the commutative diagram
\[
\xymatrix{
\ar[r]^(.3){H_*(D)} & H_*(R_{n+1}) \ar@{-->}[d]^{\overline{i_n D}} \ar[r]^{H_*(D)}  & H_*(R_n) \ar[d]_{H_*(i_n)} \ar[dr]^{H_*(D)} & & &  \\
0 \ar[r] & X[n-1] \ar[r] & H_*(Y_{n-1}) \ar[r] & H_*(R_{n-1}) \ar[r] & H_*(R_{n-2}). 
}\]

We will show that
\begin{equation}
\label{eqlemma}
\overline{i_n D} = (-1)^{\lf\frac{n-1}{2}\rf + n+ 1} \left( \sum_{\stackrel{r+s+t=n+2}{2\leq r+t+1\leq n+1}} (-1)^{r+st} 
m_{r+t+1}^X (1^{\otimes r} \otimes m_s \otimes 1^{\otimes t})\right)*1 
\end{equation}
Lemma \ref{lemmahoch} then implies that the canonical $(n+1)$-Postnikov system extends if and only if the $A_{n+1}$-structure extends to an $A_{n+2}$-structure.

To prove \eqref{eqlemma}, we need to compute $H_*(i_n D)$. We will add a null homotopic map to $i_n D$ in order to perform the computation. For $n\geq 2$, let $H_n \colon R_{n+1} \to Y_{n-1}$ be the map defined by the column vector
\[
\begin{bmatrix}
(-1)^{\lf\frac{n-1}{2}\rf+n+1} 1\otimes m_{n+2}^A\\
(-1)^{\lf\frac{n-2}{2}\rf} M_{n+1,n+1}\\
\vdots \\
(-1)^{\lf\frac{n-j}{2}\rf} M_{n+1,n+3-j} \\
\vdots \\
M_{n+1,3}
\end{bmatrix}
\]
We now compute the effect of the map
\[ i_nD + (\partial_{Y_{n-1}} H_n + H_n d) \]
on homology. We will show first that $i_n D + (\partial_{Y_{n-1}} H_n + H_n d) $ factors 
through $R_0[n-1]$: For $i\geq 2$ the $i$-th component of this map is 
\begin{eqnarray*}
&(-1)^{\lf\frac{n-i+1}{2}\rf}M_{n,n+2-i}M_{n+1,2} + \\
&\sum_{j=i}^{n} (-1)^{n-j + \lf(j-i)/2\rf + \lf(n-j)/2\rf} 
M_{j-1,j+1-i} M_{n+1,n+3-j} + \\
&+ (-1)^{\lf(n-i)/2\rf}M_{n+1,n+3-i}M_{n+1,1} 
\end{eqnarray*}
and this simplifies to 
\[ (-1)^{n+\lf\frac{i+1}{2}\rf+\lf\frac n 2\rf} \sum_{j=i}^{n+2} (-1)^{j(n+i)} 
M_{j-1,j+1-i} M_{n+1,n+3-j} \]
which up to sign is exactly the sum \eqref{summationthatvanishes} and therefore vanishes (only the
$A_{n+1}$-structure is used).

The first component of $i_n D + (\partial_{Y_{n-1}} H_n + H_n d)$ is
\begin{equation}
\label{Ansimp1}
(-1)^{\lf\frac{n-1}{2}\rf} \left( \sum_{j=2}^{n+1} (-1)^{j(n+1)} 
M_{j-1,j} M_{n+1,n+3-j} +  d (1\otimes m_{n+2}^A) + (-1)^{n+1} 1\otimes (m_{n+2}^A d) \right)
\end{equation}
Using the $A_\infty$-$H_*(A)$-module structure on $A$, a computation similar to the proof of Lemma \ref{dgmodule} shows that this formula simplifies to 
\begin{equation}
\label{boundaryeq}
(-1)^{\lf \frac{n-1}{2}\rf + n+1}\left(\sum_{\substack{2\leq r+t+1\leq n+1 \\ r+s+t=n+2}} (-1)^{r+st} m_{r+t+1}^X(1^{\otimes r}\otimes m_s \otimes 1^{\otimes t})\right)\otimes 1. 
\end{equation}
By the commutativity of diagram \eqref{diagram}, the map $H_*(R_{n+1}) \to X[n-1]$ is obtained by composing $\eqref{boundaryeq}$ with the multiplication map $X\otimes H_*(A) \to X$. This proves \eqref{eqlemma}.

It remains to prove the bijection between $A_{n+2}$-structures extending the
given $A_{n+1}$-structure and the extensions of the canonical $n$-Postnikov system when an extension exists. In that case, the $A_{n+2}$-structures are arbitrary $k$-module maps
\[ \phi \colon X \otimes H_*(A)^{\otimes (n+1)} \to X[n]. \]
On the other hand, a homotopy class of maps $j\colon R_{n+1} \to Y_n$ lifting $D$ is the same
as an $H_*(A)$-module map $H_*(R_{n+1}) \to X[n]$. Writing $i_{n+1}(\phi)$ for the lift associated
to a $k$-module map $\phi$, the formula for $i_{n+1}(\phi)$ shows that $i_{n+1}(\phi)-i_{n+1}(0)$
factors through $R_{0}[n]$ and hence (see diagram \eqref{diagram}) the $H_*(A)$-module map associated
to $i_{n+1}(\phi)-i_{n+1}(0)$ is 
\[ \phi *1 \colon H_*(R_{n+1}) \to X[n].\] 
This shows that homotopy classes of lifts of $D$ are in bijective correspondence with $k$-module maps $X \otimes H_*(A)^{\otimes(n+1)} \to X[n]$ and completes the proof.
\end{proof}

\begin{cor}
\label{mainequiv}
Any $n$-Postnikov system based on the bar resolution for $X$ is isomorphic to the canonical $n$-Postnikov system associated to an $A_{n+1}$-structure on $X$.
\end{cor}
\begin{proof}
For $n=1$ the statement is clearly true.
The result follows by induction from Theorem \ref{bijectivecorr}.
\end{proof}

\begin{lemma}
\label{lemmaakcond}
Let $(g_1,\ldots,g_k)$ be an $A_k$-map between two $A_{k+1}$-modules $X$ and $X'$. Then the square
\[ \xymatrix{ R_k \ar[d]_{g_1\otimes 1^{\otimes (k+1)}} \ar[r]^{i_k} & Y_{k-1} 
\ar[d]^{B_{k-1}(g)} \\
R_k' \ar[r]^{i_k'} & Y_{k-1}' }\]
commutes up to homotopy if and only if $(g_1,\ldots,g_k,0)$ is an $A_{k+1}$-map.
\end{lemma}
\begin{proof}
Because $g_1$ is a map of $H_*(A)$-modules, the square 
\[ \xymatrix{ R_k \ar[r] \ar[d] & R_{k-1} \ar[d] \\ R_k' \ar[r] & R_{k-1}' }
\]
commutes strictly and so the difference on homology lies in the kernel of $Y_{k-1}' \to R_{k-1}'$.
This kernel is a desuspension of $X$. We want to compute the 
map 
\[ B_{k-1}(g) i_k - i_k' (g_1\otimes 1^{\otimes (k+1)}) : H_*(R_k) \to X'[k-1]\] 
We will add a nulhomotopic map so as to make the factorization of this map through $R_0'[k-1]$
apparent. The homotopy is given by the formula
\[ H_k = \begin{bmatrix} 0 \\ (-1)^{\lf \frac k 2\rf} g_k \otimes 1^{\otimes 2} \\ \vdots \\ (-1)^{\lf\frac{k-i+2}{2}\rf}g_{k-i+2}\otimes 1^{\otimes i} \\ \vdots \\ -g_2\otimes 1^{\otimes k} \end{bmatrix}. \]
One computes that 
\[ B_{k-1}(g) i_k - i_k' (g_1\otimes 1^{\otimes (k+1)}) + ( \partial_{Y_{k-1}'} H_k +  H_k d )\]
has all components zero except the first one because $(g_1,\ldots,g_k)$ is an $A_k$-map.
 When composed with the multiplication $m_2^{X'} \colon R_0' \to X$, the first component 
yields the $(k+1)$-ary map whose vanishing is synonymous  with $(g_1,\ldots,g_k,0)$ being an $A_{k+1}$-map.
\end{proof}

Using the functor $B_n(-,H_*(A),A)$ from the last section, we can now complete the proof of
Theorem \ref{main}.
\begin{thm}
\label{mainequiv2}
Let $X$ and $X'$ be (minimal) $A_{n+1}$-$H_*(A)$-modules. There is a bijective correspondence between 
$A_{n+1}$-maps $g \colon X \to X'$, and maps between the associated canonical $n$-Postnikov systems 
based on the bar resolution.
\end{thm}
\begin{proof}
Given $g$, the desuspensions of the maps $B_k(g)$, $1\leq k \leq n$, described in Proposition \ref{defnBn} give the desired map of Postnikov systems. It is easy to check that this
assignment is injective (if two $A_n$-maps first differ on $g_k$, the induced
maps $Y_{k-1} \to Y_{k-1}'$ will not be homotopic).

The converse is proved by induction. For $n=1$, a map of Postnikov systems of the sort
described above is determined by a map of $H_*(A)$-modules $g_1 \colon X \to X'$ and a map
\[ f_1 \colon Y_1 \to Y_1' \]
such that 
\[ \xymatrix{ Y_1 \ar[d]^{f_1} \ar[r]& R_1  \ar[d]^{g_1\otimes 1\otimes 1} \ar[r] & R_0 \ar[d]^{g_1\otimes 1} \\ 
Y_1' \ar[r] & R_1' \ar[r] & R_0' }\]
is a map of triangles. Thus $f_1$ can be represented by a matrix
\[ \begin{bmatrix} 
g_1\otimes 1 & \tilde{g_2} \\
0 & g_1\otimes 1 \otimes 1 
\end{bmatrix}.\]
Since $g_1$ is a map of $H_*(A)$-modules, the matrix above with $\tilde g_2 = 0$ also 
defines a map of triangles. The difference between these two matrices factors as 
\[ Y_1 \to R_1 \to R_0'[1] \to Y_1' \]
There is a unique representative for the homotopy class of the middle map of the form 
$g_2\otimes 1$ and therefore $f_1$ has a unique representative of the form 
\[ \begin{bmatrix} 
g_1\otimes 1 & -g_2\otimes 1 \\
0 & g_1\otimes 1 \otimes 1 
\end{bmatrix}.\]
The only requirement for $(g_1,g_2)$ to be a map of $A_2$-modules is that $g_1$ commutes
with the multiplication. This completes the proof for $n=1$.

Suppose given a map of $k$-Postnikov systems based on the bar construction. By induction
we know that there is a unique map $g=(g_1,\ldots,g_k)$ of $A_k$-modules
such that $Y_{j} \to Y_{j}'$ is $B_j(g)$ for $j\leq k-1$.

There is a commutative square
\[ \xymatrix{ Y_k \ar[r] \ar[d]^{f_k} & R_k \ar[r]^{i_k} \ar[d]^{g_1\otimes 1^{\otimes (k+1)}} &  Y_{k-1} \ar[d]^{B_{k-1}(g)} \\ Y'_k \ar[r] & R_k' 
\ar[r]_{i_k'} & Y_{k-1}' } \]

By Lemma \ref{lemmaakcond},  $(g_1,\ldots,g_k,0)$ is an $A_{k+1}$-module map.
Let $d= f_k - B_{k+1}(g_1,\ldots,g_k,0)$. This difference factors
as
\[ Y_k \to R_k \to Y_{k-1}'[1] \to Y_k'. \]
The homotopy class of a map from $R_k$ is determined by its effect on homology. Since
$R_k \to Y_k'$ factors through $Y_{k-1}'[1]$, it is 0 along $R_k'$ and hence its
image lies in the kernel of the map $Y_k' \to R_k'$ which is $X'[k]$.

Therefore it factors through $R_0'[k]$ up to homotopy and the homotopy class is therefore
represented uniquely by a map of the form $(-1)^{\lf \frac{k+1}{2}\rf} g_{k+1}\otimes 1$. We conclude that the homotopy class
of $f_k$ is equal to that of $B_{k+1}(g_1,\ldots,g_{k+1})$ (note that any choice of $g_{k+1}$
will give an $A_{k+1}$-map).
\end{proof}
The previous Theorem show that the functor sending an $A_{n+1}$-structure to its
canonical $n$-Postnikov system is full and faithful. Corollary \ref{mainequiv} asserts
that this functor is essentially surjective hence it is an equivalence of categories.
This completes the proof of Theorem \ref{main}.
 
\begin{remark}
It follows from Theorem \ref{mainequiv2} that if $g: X \to X'$ is an $A_k$-map such
that $g_1$ is an isomorphism then $g$ is an isomorphism.
\end{remark}

Let $X$ be an $H_*(A)$-module. The \emph{moduli groupoid of $A_{n+1}$-structures on $X$}
is the groupoid with objects $A_{n+1}$-module structures on $X$ and quasi-isomorphisms
$g$ between them with $g_1=\id$. Note that this is equivalent to the groupoid of $A_{n+1}$-modules $X'$ together with an isomorphism of $H_*(A)$-modules $X' \to X$.

\begin{cor}
\label{mainequiv3}
The moduli groupoid of $A_{n+1}$-structures on $X$ is equivalent to the groupoid of $n$-Postnikov
systems for $X$ based on the bar resolution and isomorphisms which are the identity on the bar
resolution.
\end{cor}

\appendix

\section{Relation between realization of Postnikov systems and chain complexes}
\label{appen}
In this appendix we explain the relation between Postnikov systems and rigidifying complexes
in a homotopy category (see \cite{DKS} for the general theory of realizing diagrams). We explain this in the setting of model categories (see \cite{Ho}). The model category $\cat C$ which is relevant for this paper is the category of differential graded modules over a DGA $A$ with the standard projective model structure (see for example \cite{SS}).

\begin{defn}
Let $\cat C$ be a pointed category. A \emph{chain complex} in $\cat C$ is a sequence
of maps in $\cat C$
\[  \cdots \xrightarrow{d} C_n \xrightarrow{d} C_{n-1} \xrightarrow{d} \cdots \xrightarrow{d} C_0 \]
such that $d d = *$.
\end{defn}

\begin{defn} 
If $\cat C$ is a pointed model category, a \emph{Postnikov system} is a 
commutative diagram in $\Ho(\cat C)$
\[\xymatrix{
& Y_n \ar[d]_{j_n} & Y_{n-1} \ar[d]_{j_{n-1}} & & Y_1 \ar[d]_{j_1} \\
\cdots \ar[r] & C_n \ar[ur]^{i_{n-1}} \ar[r] & C_{n-1} \ar[r] & \cdots & C_1 \ar[r]_{i_0} & C_0
}\]
where for each $k$, the sequence
\[ Y_k \xrightarrow{j_k} C_k \xrightarrow{i_{k-1}} Y_{k-1} \]
is a homotopy fiber sequence (we set $Y_0=C_0$). 

An \emph{$m$-Postnikov system} is a diagram as above but with objects only those $Y_i$ and $C_i$ 
where $i\geq m+1$.
\end{defn}
Note that in a Postnikov system, $C_\bullet$ is a chain complex in $\Ho(\cC)$.

\begin{defn}
Let $\cat C$ be a model category, $\pi \colon \cat C \to \Ho(\cat C)$ be the canonical functor
and $I$ be a small category. A diagram $F:I \to \Ho(\cat C)$ is \emph{realizable} if there exists a diagram $\tilde F: I \to \cat C$ together with a natural isomorphism $\phi: \pi \tilde F
\to F$. The diagram $\tilde F$ is then called a \emph{realization} of $F$.

If $\cat C$ is pointed we say that a diagram is \emph{strictly realizable} if $F(\alpha)=*$
implies that $\tilde F(\alpha)=*$ and $\tilde F$ is then called a strict realization of 
$F$.
\end{defn}

\begin{prop}
\label{comp-eq}
Let $\cat C$ be a pointed model category. Let $C_\bullet$ be a chain complex in $\Ho(\cat C)$. Then
the following are equivalent:
\begin{enumerate}[(i)]
\item $C_\bullet$ is strictly realizable,
\item $C_\bullet$ extends to a Postnikov system,
\end{enumerate}
\end{prop}
\begin{proof}
\emph{(i) $\Rightarrow$ (ii)}:
Replacing $C_\bullet$ if necessary by an isomorphic complex we may assume that
\[ \cdots \to C_n \xrightarrow{d_{n-1}} C_{n-1} \to \cdots \]
is a chain complex in $\cat C$ projecting to $C_\bullet$.

Replacing the map $C_1 \xrightarrow{d_0} C_0$ by a fibration we obtain a 
diagram
\[
\xymatrix{ C_3 \ar[r]^{d_2} &
C_2 \ar[r]^{d_1} \ar@{-->}[d]_{i_1} & C_1 \ar[r]^{d_0} \ar[d]_{\sim} & C_0 \\
& Y_1 \ar[r] \ar@/_/[rr] & C_1' \ar@{->>}[ur] & \ast \ar[u] 
}\]
where $Y_1$ is the homotopy fiber of $C_1' \to C_0$.
Since $d_{0} d_{1}=*$, there is a canonical factorization $C_2 \llra{i_1} Y_1$.
Furthermore, the composite 
\[ C_3 \llra{d_2} C_2 \llra{i_1} Y_1\]
is the zero map since its composite with the map $Y_1 \to C_1'$ is 
zero by construction of $i_1$. 

We may apply the same procedure to the sequence of maps 
\[ \cdots \to C_3 \xrightarrow{d_{2}} C_{2} \xrightarrow{i_1} Y_1 \]
and continuing inductively we obtain a Postnikov system which we denote 
by $P(C_\bullet)$.
 
This construction is clearly functorial so we have defined a functor
\begin{equation}
\label{functorP}
P \colon \mathcal{CC} \longrightarrow \mathcal{PS}
\end{equation}
from the category of chain complexes in $\cat C$ to the category of Posnikov systems in 
$\Ho(\cat C)$ which sends weak equivalences to isomorphisms. 

\noindent
\emph{(ii)$\Rightarrow$ (i)}: Let
\[\xymatrix{
& Y_n \ar[d]_{j_n} & Y_{n-1} \ar[d]_{j_{n-1}} & & Y_1 \ar[d]_{j_1} \\
\cdots \ar[r] & C_n \ar[ur]^{i_{n-1}} \ar[r] & C_{n-1} \ar[r] & \cdots & C_1 \ar[r]_{i_0} & C_0
}\]
be a Postnikov system in $\Ho(\cat C)$. 

We will write $\overline f$ for an arbitrary representative of the map $f \in \Ho(\cat C)$
and $[\psi]$ for the homotopy class of $\psi \in \cat C$.

First note we can assume that all the objects $Y_k$ and $C_k$ are fibrant and cofibrant.
 We will construct a chain complex $\tilde C_\bullet$ in $\cat C$ lifting $C_{\bullet}$ inductively. 

Let $\tilde C_0 = C_0$. Let
\[ \xymatrix{  C_1 \ar[d]_{\overline d_0} \ar[r]^{\phi_1} & \tilde{C}_1 \ar@{->>}[dl]^{\tilde d_0} \\
C_0 & }\]
be a factorization of ${\overline d_0}$ into a trivial cofibration followed by a fibration and 
\[ \tilde Y_1 \xrightarrow{\tilde{j_1}} \tilde C_1 \]
be the inclusion of the fiber of $\tilde d_0$. Since $Y_1$ is the homotopy fiber of $d_0$, there is an isomorphism $\psi_1 \colon Y_1 \to \tilde Y_1$ such that
\[  \xymatrix{ 
Y_1 \ar[d]_{j_1} \ar[r]^{\psi_1} & \tilde{Y_1} \ar[d]_{[\tilde j_1]} \\
C_1 \ar[r]^{[\phi_1]} & \tilde C_1 }\]
commutes.

Now factor $\overline{\psi_1}\ {\overline i_1} \colon C_2 \to \tilde Y_1$ (which exists because $C_2$ is cofibrant
and $\tilde Y_1$ is fibrant) as a trivial cofibration $\phi_2$ followed by a
fibration $\tilde i_1$. We get a comutative diagram
\[ \xymatrix{
C_2 \ar[r]^{\phi_2 } \ar[d]_{\overline{i_1}} & \tilde C_2 \ar@{->>}[d]_{\tilde i_1} \\
Y_1 \ar[r]_{\overline{\psi_1}} & \tilde Y_1  } \]
Let $\tilde d_1 = \tilde j_1 \tilde i_1$. Since $\tilde Y_1$ is the fiber of $\tilde d_0$ it follows that
the composite $\tilde d_0 \tilde d_1$ is the zero map.

Let $\tilde j_2 \colon \tilde Y_2 \to C_2$ denote the inclusion of the fiber of $\tilde i_1$.
Since $Y_2 \to C_2 \to Y_1$ is a fiber sequence, there is an isomorphism $\psi_2\colon Y_2 \to \tilde Y_2$
in $\Ho(\cat C)$ such that 
\[ \xymatrix{ 
Y_2 \ar[d]_{j_2} \ar[r]^{\psi_2} & \tilde Y_2 \ar[d]_{[\tilde j_2]} \\
C_2 \ar[r]^{[\phi_2]} & \tilde C_2 }
\]
commutes and we can proceed inductively to obtain a realization $\tilde C_\bullet$ of $C_\bullet$.
\end{proof}

\begin{remark}
The statements in Proposition \ref{comp-eq} are equivalent to the vanishing of the \emph{Toda brackets} $\langle d_0,\ldots,d_n \rangle$ for all $n\geq 2$. The Toda bracket can be defined in several different
ways. We use the following definition: $\langle d_0,\ldots,d_n\rangle$ is a subset (possibly empty) of $\Ho(\cC)(C_{n+1},\Omega^{n-1}C_0)$ consisting of all possible lifts $\phi$ 
in diagrams of the form \eqref{postsys}, for all choices of $n$-Postnikov systems extending
$C_n \to \cdots \to C_0$.
\begin{equation}
\label{postsys}
\xymatrix{
& & \Omega^{n-1} C_0 \ar[d]&  & & \\
& & \Omega^{n-2} Y_1 \ar[d]&  & & \\
& & \vdots \ar[d] & &  \\
& & \Omega{Y^{n-2}}\ar[d] & &  &\\
& & Y_{n-1} \ar[d]^{j_{n-1}} & & Y_1 \ar[d] & \\    
C_{n+1} \ar[r]_{d_n} \ar@{-->}[uuuuurr]^{\phi} & C_n \ar[ur]^{i_{n-1}} \ar[r]_{d_{n-1}} & C_{n-1} \ar[r]_{d_{n-2}} & \cdots 
 \ar[r] & C_1 \ar[r]_{d_0} & C_0  
}
\end{equation}
$\Omega^j$ denotes the $j$-th iterate of the loop functor and the vertical maps belong to the homotopy fiber sequences which end in $Y_k \to C_k \to Y_{k-1}$ (see \cite[Chapter 6]{Ho}).

We say that a Toda bracket vanishes if it contains the zero map.
It is clear that the $n$-Postnikov system in \eqref{postsys} extends one stage if and only
if $\phi$ can be chosen to be zero. Thus, an $n$-Postnikov system
encodes the vanishing of the Toda bracket of the maps in the underlying chain complex.

The higher order cohomology operations in \cite[16.3]{Ma} are defined as Toda brackets
with the above definition. The definition of Toda bracket in \cite[IV.1]{Wh} is very similar. Whitehead works in a stable setting where cofiber and fiber sequences are equivalent. To define the Toda bracket of $<d_0,\ldots,d_n>$ he considers all possible diagrams \footnote{It is easy to check using the limited naturality of triangles that in the definition of Toda bracket in \cite{Wh} we may assume that either the map $i_0$ or $j_n$ are the identity and we are taking $j_n$ to be the identity.}
\[ 
\xymatrix{ 
& X_{n-1} \ar[rd] & X_{n-2} & & X_1 \ar[rd]^{i_1} & \\ 
X_n =C_{n+1} \ar[r]_{d_n} & C_n \ar[u] \ar[r]_{d_{n-1}} & C_{n-1} \ar[u] \ar[r] & \cdots \ar[r] & C_2 \ar[u] \ar[r]_{d_1} & C_1
}
\]
where the sequences $X_i \to C_i \to X_{i-1}$ are cofiber sequences and defines the Toda bracket to be the set of all possible extensions of $d_0 i_1$ along
\[ X_1 \to \Sigma X_2 \to \cdots \to \Sigma^{n-1}C_n \]
where $\Sigma$ denotes the suspension functor.

It is possible to check that our definition and Whitehead's agree by exhibiting both sets
as certain choices of $(n-1)$-spheres $\partial\Delta^n \subset \Hom(C_n,C_0)$ in the homotopy function complex from $C_n$ to $C_0$. For more on this perspective, see \cite[Examples 3.10,3.20]{BM}.
\end{remark}

\begin{prop}
\label{rigidifications}
Let $\cat C$ be a pointed model category, $\mathcal{CC}$ be the category of length $n$ chain complexes in $\cat C$ (with $n\leq \infty$) and $\mathcal{PS}$ be the category of $n$-Postnikov systems in $\Ho(\cat C)$. Then the functor
\[ P \colon \mathcal{CC} \to \mathcal{PS} \]
(see \eqref{functorP}) induces a bijection from weak equivalence classes in 
$\mathcal{CC}$ to isomorphism classes of objects in $\mathcal{PS}$.
\end{prop}
\begin{proof}
A chain complex realizing a Postnikov system will, by definition, realize any equivalent Postnikov system so we have already proved in Proposition \ref{comp-eq} that the functor $P$ is
essentially surjective.

On the other hand using the homotopy lifting property for fibrations, the 
construction of the chain complex from the Postnikov system in Proposition \ref{comp-eq} will also 
yield lifts of isomorphisms between Postnikov systems to weak equivalences between chain complexes in $\cat C$ (because the $\tilde C_i$ are fibrant and cofibrant, the $\tilde Y_i$ are the actual fibers of maps and the maps $\tilde C_{i+1} \to \tilde Y_i$ are fibrations).
\end{proof}

\end{document}